\documentclass[preprint,3p]{elsarticle}

\usepackage{amssymb}
\usepackage{amsmath}
\usepackage{amsfonts}
\usepackage{amsthm}
\usepackage{indentfirst}
\usepackage{bm}
\usepackage[latin1]{inputenc}

\newtheorem{lemma}{Lemma}[section]
\newtheorem{theorem}{Theorem}[section]

\newcommand{\field}[1]{\mathbb{#1}}

\newcommand{\PP}{\field{P}}
\newcommand{\NN}{\field{N}}
\def\C{{\mathcal{C}}}

\journal{}

\begin{document}

\begin{frontmatter}



\title{An error estimate for the Gauss-Jacobi-Lobatto quadrature rule}

\author{Concetta Laurita\fnref{Tel}}
\address{Department of Mathematics, Computer Science and Economics, University of Basilicata, Viale dell'Ateneo Lucano n. 10, 85100 Potenza, ITALY}
\ead{concetta.laurita@unibas.it}
\fntext[Tel]{Tel: +390971205846}

\begin{abstract}
An error estimate for the Gauss-Lobatto quadrature formula for integration over
the interval $[-1, 1]$, relative to the Jacobi weight function $w^{\alpha,\beta}(t)=(1-t)^\alpha(1+t)^\beta$, $\alpha,\beta>-1$, is obtained.
This estimate holds true for functions belonging to some Sobolev-type subspaces of the weighted space $L_{w^{\alpha,\beta}}^1([-1,1])$.
\end{abstract}

\begin{keyword}
Gauss-Lobatto formula; Jacobi weight function
\MSC[2020] 65D30, 65D32
\end{keyword}
\end{frontmatter}
%
%
\section{Introduction}
For a Jacobi weight function $w^{\alpha,\beta}(t)=(1-t)^\alpha(1+t)^\beta$, $\alpha,\beta>-1$, on the interval $[-1,1]$ we consider the $(n+2)$-point Gauss-Lobatto rule
\begin{equation} \label{LobattoRule}
\int_{-1}^1 f(t)w^{\alpha,\beta}(t)dt=w_{0}f(-1)+\sum_{k=1}^n w_{k}f(t_{k})+w_{n+1}f(1)+e_n(f)
\end{equation}
which is exact for polynomials of degree at most $2n+1$, i.e.
\begin{equation*}
e_n(f)=0, \qquad \forall f \in \PP_{2n+1}
\end{equation*}
being $\PP_n$ the set of all algebraic polynomials on $[-1,1]$ of degree at most $n$. \newline
It is well known that the interior quadrature nodes $t_{k}$, $k=1,\ldots,n$, are the zeros of the Jacobi polynomial of degree $n$ orthonormal with respect to the Jacobi weight
$w^{\alpha+1,\beta+1}(t)=(1-t)^2w^{\alpha,\beta}(t)$. The weights of the formula \eqref{LobattoRule} are given by
\begin{equation} \label{QuadrWeights}
w_{k}=\int_{-1}^1l_k(t)w^{\alpha,\beta}(t)dt, \qquad k=0,1,\ldots,n+1
\end{equation}
where, setting from now on $t_0=-1$ and $t_{n+1}=1$, $l_k(t)$ denotes the $(k+1)$-th Lagrange fundamental polynomial associated to the system of nodes $\left\{t_0,t_1,\ldots,t_n,t_{n+1}\right\}$. \newline
By standard arguments, it can be easily proved that the quadrature error $e_n(f)$ satisfies the following estimate
\begin{equation*}
|e_n(f)| \leq 2 \left(\int_{-1}^1 w^{\alpha,\beta}(t)\,dt\right) E_{2n+1}(f)_{\infty}, \qquad \forall f \in C([-1,1]),
\end{equation*}
where
\[E_n(f)_{\infty}=\inf_{P\in \PP_n}\|f-P\|_{\infty}\]
denotes the error of best approximation of a function $f\in C([-1,1])$ by means of polynomials of degree at most $n$ with respect to the uniform norm. \newline
The aim of the present paper is to provide a new error estimate for less regular functions belonging to some Sobolev-type subspaces of the weighted space $L_{w^{\alpha,\beta}}^1([-1,1])$. A similar estimate is proved in \cite{MMbook} for the classical Gauss-Jacobi quadrature formula, in \cite{FermoLaurita} for the Gauss-Lobatto rule with respect to the Legendre weight $w^{0,0}$ and, more recently, in \cite{Laurita} for the Gauss-Radau formula with respect to a general Jacobi weight $w^{\alpha,\beta}$.

\section{Notation and preliminary results}
\subsection{Notation}
For a general weight function $w(t)$ on $[-1,1]$  and  $1\leq p<+\infty$, let $L^p_{w}$ denote the weighted space of all real-valued measurable functions $f$ on $[-1,1]$ such that
\begin{equation*}
\|f\|_{L^p_{w}}=\|f w \|_p=\left(\int_{-1}^1 |f(t)w(t)|^pdt\right)^\frac 1p<+\infty,
\end{equation*}
and let $W_r^p(w)$ be the following weighted Sobolev-type subspaces of $L^p_{w}$
\begin{equation*}
W_r^p(w)=\left\{f \in L^p_{w} : \ f^{(r-1)}\in AC(-1,1), \ \|f^{(r)}\varphi^r w \|_p<+\infty\right\},
\end{equation*}
where $r \in \NN$, $r \geq 1$, $\varphi(t)=\sqrt{1-t^2}$ and $AC(-1,1)$ is the collection of all functions which are absolutely continuous on every closed subset of $(-1,1)$, equipped with the norm
\[\|f\|_{W_r^p(w)}=\|f w\|_p+\|f^{(r)}\varphi^r w\|_p.\]
For a function $f \in L_w^p$, the error of the best approximation of $f$ in $L_w^p$ by polynomials of degree at most $n$ is defined as
\[E_n(f)_{w,p}=\inf_{P\in \PP_n}\|f-P\|_{L^p_{w}}.\]
Fixed a Jacobi weight $w^{\gamma,\delta}(t)=(1-t)^{\gamma}(1+t)^{\delta}$, $\gamma,\delta>-1$, we will denote by $x_{n,k}^{\gamma,\delta}$ and $\lambda_{n,k}^{\gamma,\delta}$, $k=1,\ldots,n$, the nodes and coefficients of the corresponding $n$--point Gauss-Jacobi quadrature rule on $[-1,1]$ and by
\[p_n^{\gamma,\delta}(t)=\gamma_n^{\gamma,\delta}t^n+\mathrm{lower \, degree \, terms}, \qquad \gamma_n^{\gamma,\delta}>0,\]
the Jacobi polynomial of degree $n$ orthonormal w.r.t. $w^{\gamma,\delta}(t)$ having positive leading coefficient.\newline
In the sequel $\C$ will denote a positive constant which may assume different values in different formulas. We write $\C =\C(a,b,...)$ to say that $\C$ is dependent on the parameters $a,b,....$ and $\C \neq \C(a,b,...)$ to say that $\C$ is independent of them.
Moreover, we will write $A \sim B$, if there exists a positive constant $\C$ independent of the parameters of $A$ and $B$ such that
$1/\C \leq A/B \leq \C$.

\subsection{Preliminary results}
It is well known that for functions $f$ belonging to $W_1^p(w)$, the following Favard inequality
\begin{equation}\label{Favard1}
E_n(f)_{w,p}\leq \frac \C{n}E_{n-1}(f')_{\varphi w,p},
\end{equation}
is fulfilled for a positive constant $\C$ independent of $n$ and $f$  (see, for example, \cite[(2.5.22), p. 172]{MMbook}). By iteration of inequality \eqref{Favard1}, it follows that, for $f\in W_r^p(w)$, $r\geq 1$,  the estimate
\begin{equation}\label{Favard}
E_n(f)_{w,p}\leq \frac \C{n^r} E_{n-r}(f^{(r)})_{\varphi ^r w,p}, \qquad \C\neq\C(n,f)
\end{equation}
holds true. \newline
Let us recall that the knots $x_{n,k}^{\gamma,\delta}$ and Christoffel numbers $\lambda_{n,k}^{\gamma,\delta}$ of the Gaussian quadrature formula corresponding to the Jacobi weight $w^{\gamma,\delta}$  satisfy the following properties (see, for instance, \cite[(4.2.4), p. 249]{MMbook})
\begin{equation} \label{Deltaxnk}
x_{n,k+1}^{\gamma,\delta}-x_{n,k}^{\gamma,\delta} \sim  \frac{\sqrt{1-t^2}}{n}, \qquad  x_{n,k}^{\gamma,\delta} \leq t \leq x_{n,k+1}^{\gamma,\delta},
\end{equation}
and (see \cite[(14), p.673]{Nevai2})
\begin{equation} \label{lambdaktilde}
\lambda_{n,k}^{\gamma,\delta} \sim \frac{\sqrt{1-\left(x_{n,k}^{\gamma,\delta}\right)^2}}{n} w^{\gamma,\delta}(x_{n,k}^{\gamma,\delta})
\end{equation}
uniformly for $1 \leq k \leq n$, $n \in \NN$. Moreover, for the orthonormal polynomials $\left\{p_n^{\gamma,\delta}(t)\right\}_n$ one has that  (see \cite[(12.7.2), p. 309]{Szego})
\begin{equation} \label{gamman}
\frac{\gamma_{n}^{\gamma,\delta}}{\gamma_{n-1}^{\gamma,\delta}} \sim 1, \qquad \mathrm{as} \quad n \rightarrow \infty,
\end{equation}
and (see, for instance, \cite[p. 170]{Nevai}) the equality
\begin{equation} \label{Nevaiequality}
\frac{1}{p_{n-1}^{\gamma,\delta}\left(x_{n,k}^{\gamma,\delta}\right)}=
\frac{\gamma_{n-1}^{\gamma,\delta}}{\gamma_{n}^{\gamma,\delta}}\lambda_{n,k}^{\gamma,\delta}\left(p_n^{\gamma,\delta}\right)'
\left(x_{n,k}^{\gamma,\delta}\right)
\end{equation}
holds true. Furthermore (see \cite[Corollary 9.34, p. 171]{Nevai}
\begin{equation} \label{pnendpoints}
p_n^{\gamma,\delta}(1) \sim n^{\alpha+\frac{1}{2}} \qquad \left|p_n^{\gamma,\delta}(-1)\right| \sim n^{\beta+\frac{1}{2}},
\end{equation}
uniformly for $n \in \NN$ and, more generally, \cite[(4.4.49)-(4.2.30), p. 255]{MMbook})
\begin{equation} \label{pnendintervals1}
\left|p_n^{\gamma,\delta}(t)\right| \sim n^{\alpha+\frac{1}{2}}, \qquad 1-\frac{\C}{n^2}\leq t \leq 1,
\end{equation}
\begin{equation} \label{pnendintervals2}
\left|p_n^{\gamma,\delta}(t)\right| \sim n^{\beta+\frac{1}{2}}, \qquad  -1 \leq t \leq -1+\frac{\C}{n^2}.
\end{equation}

\section{Main results}
\begin{lemma} \label{lemma}
The nodes and the weights of the Gauss-Lobatto quadrature formula \eqref{LobattoRule} satisfy the following relations
\begin{subequations}
\begin{align}
w_0 \displaystyle & \leq \C \, \Delta t_0 w^{\alpha,\beta}(t_1), \label{weight1tilde} \\
w_k \displaystyle &  \sim \left \{ \begin{array}{lcl} \Delta t_k \, w^{\alpha,\beta}(t_k),  & \quad &  k=1,\ldots,n-1\\
\displaystyle  \Delta t_{k-1} \, w^{\alpha,\beta}(t_k) & \quad & k=n \end{array} \right., \label{weight2tilde}\\
w_{n+1} \displaystyle & \leq  \C \, \Delta t_n w^{\alpha,\beta}(t_n), \label{weight3tilde}
\end{align}
\end{subequations}
where $\Delta t_k=t_{k+1}-t_k$, $k=0,1,\ldots,n$ and  $\C \neq \C(n)$.
\end{lemma}
\begin{proof}
First, let us observe that the weights of the formula \eqref{LobattoRule}, given in \eqref{QuadrWeights}, have  the following alternative representation
\begin{subequations}
\begin{align}
w_0 \displaystyle & = \frac{1}{1-t_0}\int_{-1}^1\frac{p_n^{\alpha+1,\beta+1}(t)}{p_n^{\alpha+1,\beta+1}(t_0)}w^{\alpha+1,\beta}(t) dt,\label{weight1} \\
w_k \displaystyle & =\frac{1}{\left(1-t_k^2\right)} \int_{-1}^1\frac{p_n^{\alpha+1,\beta+1}(t)}{\left(t-t_k\right)\left(p_n^{\alpha+1,\beta+1}\right)'(t_k)}w^{\alpha+1,\beta+1}(t) dt, \qquad k=1,\ldots,n,   \label{weight2} \\
w_{n+1} \displaystyle & =  \frac{1}{1+t_{n+1}}\int_{-1}^1\frac{p_n^{\alpha+1,\beta+1}(t)}{p_n^{\alpha+1,\beta+1}(t_{n+1})}w^{\alpha,\beta+1}(t) dt.\label{weight3}
\end{align}
\end{subequations}
For $k=1,\ldots,n$, since
\begin{equation*}
w_{k}=\frac{\lambda_{n,k}^{\alpha+1,\beta+1}}{1-\left(x_{n,k}^{\alpha+1,\beta+1}\right)^2},
\end{equation*}
using \eqref{lambdaktilde} and \eqref{Deltaxnk} for $\gamma=\alpha+1$ and $\delta=\beta+1$, we can deduce that
\begin{equation*}
w_{k} \sim  \left \{\begin{array}{lcl} \displaystyle w^{\alpha,\beta}(t_k)\Delta t_k & \quad & k=1,\ldots,n-1 \\
\displaystyle w^{\alpha,\beta}(t_k)\Delta t_{k-1}  & \quad & k=n\end{array} \right.
\end{equation*}
i.e. \eqref{weight2tilde}. Now, in order to prove \eqref{weight1tilde}, let us observe that, since
\begin{equation*}
\int_{-1}^1\left[\frac{p_n^{\alpha+1,\beta+1}(t)}{p_n^{\alpha+1,\beta+1}(t_0)}-\left(\frac{p_n^{\alpha+1,\beta+1}(t)}{p_n^{\alpha+1,\beta+1}(t_0)}
\right)^2\right] w^{\alpha+1,\beta}(t)dt=0,
\end{equation*}
we can rewrite the first coefficient $w_0$ in \eqref{weight1} as follows
\begin{equation*}
w_0 \displaystyle = \frac{1}{2}\left[\frac{\left(p_n^{\alpha+1,\beta+1}\right)'(t_1)}{p_n^{\alpha+1,\beta+1}(t_0)} \right]^2\int_{-1}^1\left[\frac{p_n^{\alpha+1,\beta+1}(t)}{(t-t_1)\left(p_n^{\alpha+1,\beta+1}\right)'(t_1)}\right]^2
\frac{(t-t_1)^2 }{1+t}w^{\alpha+1,\beta+1}(t) dt.
\end{equation*}
Being \((t-t_1)^2/(1+t) \leq \C\)
and using \eqref{Nevaiequality} (with $\gamma=\alpha+1$ and $\delta=\beta+1$) for $k=1$ (we recall that in our notation $x_{n,1}^{\alpha+1,\beta+1} \equiv t_1$ ), we get
\begin{eqnarray*}
w_0 & \leq & \C \left[\frac{\left(p_n^{\alpha+1,\beta+1}\right)'(t_1)}{p_n^{\alpha+1,\beta+1}(t_0)} \right]^2 \lambda_{n,1}^{\alpha+1,\beta+1}\\
& = & \C \frac{1}{\left[p_n^{\alpha+1,\beta+1}(t_0)\right]^2} \left(\frac{\gamma_{n}^{\alpha+1,\beta+1}}{\gamma_{n-1}^{\alpha+1,\beta+1}}\right)^2
\frac{1}{\lambda_{n,1}^{\alpha+1,\beta+1}\left[p_{n-1}^{\alpha+1,\beta+1}\left(t_1\right)\right]^2}.
\end{eqnarray*}
Now, in virtue of \eqref{lambdaktilde}, it is
\begin{equation} \label{lambda1tilde}
\lambda_{n,1}^{\alpha+1,\beta+1} \sim \frac{\left(1-x_{n,1}^{\alpha+1,\beta+1}\right)^{\alpha+\frac{3}{2}}\left(1+x_{n,1}^{\alpha+1,\beta+1}\right)^{\beta+\frac{3}{2}}}{n} \sim \frac{1}{n^{2\beta+4}}.
\end{equation}
Taking into account \eqref{lambda1tilde}, and \eqref{gamman}, \eqref{pnendpoints} and \eqref{pnendintervals2} (all applied with $\gamma=\alpha+1$ and $\delta=\beta+1$), we deduce the following estimate of $w_0$
\begin{equation} \label{estw01}
w_0 \leq \C \frac{n^{2\beta+4}}{n^{4\beta+6}} = \frac{\C}{n^{2\beta+2}}.
\end{equation}
On the other hand,
\begin{equation*}
\Delta t_0 w^{\alpha,\beta}(t_1)=\left(1-x_{n,1}^{\alpha+1,\beta+1}\right)^{\alpha}\left(1+x_{n,1}^{\alpha+1,\beta+1}\right)^{\beta+1} \sim \frac{1}{n^{2\beta+2}}
\end{equation*}
which, combined with \eqref{estw01}, gives \eqref{weight1tilde}. In order to prove \eqref{weight3tilde}, proceeding in an analogous way, we start writing $w_{m+1}$ as follows
\begin{equation*}
w_{m+1}=\frac{1}{2}\left[\frac{\left(p_n^{\alpha+1,\beta+1}\right)'(t_n)}{p_n^{\alpha+1,\beta+1}(t_{n+1})} \right]^2\int_{-1}^1\left[\frac{p_n^{\alpha+1,\beta+1}(t)}{(t-t_n)\left(p_n^{\alpha+1,\beta+1}\right)'(t_n)}\right]^2
\frac{(t-t_n)^2 }{1-t}w^{\alpha+1,\beta+1}(t) dt.
\end{equation*}
Then, using $(t-t_n)^2/(1-t)\leq \C$, \eqref{lambdaktilde} and \eqref{Nevaiequality} for $\gamma=\alpha+1$, $\delta=\beta+1$ and $k=n$, \eqref{gamman}, \eqref{pnendpoints} and \eqref{pnendintervals1} also with $\gamma=\alpha+1$ and $\delta=\beta+1$, we deduce the estimate
\begin{eqnarray*}
w_{n+1}  \leq  \C \left[\frac{\left(p_n^{\alpha+1,\beta+1}\right)'(t_n)}{p_n^{\alpha+1,\beta+1}(t_{n+1})} \right]^2 \lambda_{n,n}^{\alpha+1,\beta+1} \sim  \frac{1}{n^{2\alpha+2}}
\end{eqnarray*}
with $\C \neq \C(n)$. Since
\begin{equation*}
\Delta t_n w^{\alpha,\beta}(t_n) \sim \left(1-x_{n,n}^{\alpha+1,\beta+1}\right)^{\alpha+1}\left(1+x_{n,n}^{\alpha+1,\beta+1}\right)^{\beta} \sim \frac{1}{n^{2\alpha+2}}
\end{equation*}
we can, finally, conclude that also \eqref{weight3tilde} holds true.
\end{proof}

Using the previous lemma we are able to prove our main result.

\begin{theorem} \label{QuadrErrorEstTheo}
For $f \in W_r^1(w^{\alpha,\beta})$, $r\geq 1$, the error of the Gauss-Lobatto quadrature formula \eqref{LobattoRule} satisfies the following estimate
\begin{equation}\label{QuadrErrorEst}
|e_n(f)| \leq \frac{\C}{(2n)^r}E_{2n+1-r}(f^{(r)})_{\varphi^r w^{\alpha,\beta},1}
\end{equation}
where $\varphi(t)=\sqrt{1-t^2}$ and $\C \neq \C(f,n)$ is a positive constant.
\end{theorem}
\begin{proof}
We start by proving the estimate \eqref{QuadrErrorEst} in the case $r=1$. First, we are going to show that
\begin{equation} \label{firstinequality}
\sum_{k=0}^{n+1} w_k|f(t_k)| \leq \C \|f w^{\alpha,\beta}\|_1+\frac{\C}{n}\int_{-1}^1|f'(t)|\varphi(t)w^{\alpha,\beta}(t)dt.
\end{equation}
Taking into account \eqref{weight1tilde}, \eqref{weight2tilde} and \eqref{weight3tilde} we have
\begin{eqnarray*}
\sum_{k=0}^{n+1} w_k|f(t_k)|& \leq &
 \C \left[\right.
 \Delta t_0 w^{\alpha,\beta}(t_1)|f(t_0)|+ \sum_{k=1}^{n-1}\Delta t_k w^{\alpha,\beta}(t_k)|f(t_k)|+ \\
& + & \Delta t_{n-1} \, w^{\alpha,\beta}(t_n)|f(t_n)| + \Delta t_n w^{\alpha,\beta}(t_n)|f(t_{n+1})| \left.\right].
\end{eqnarray*}
In virtue of the following inequality
\begin{equation*}
\left. \begin{array}{l} (b-a)|f(a)|\\ (b-a)|f(b)|\end{array}\right\} \leq \int_a^b |f(t)|dt+(b-a)\int_a^b |f'(t)|dt,
\end{equation*}
it follows that
\begin{eqnarray*}
\sum_{k=0}^{n+1} w_k|f(t_k)|& \leq & \C \left[w^{\alpha,\beta}(t_1)\left(\int_{t_0}^{t_1}|f(t)|dt+\Delta t_0 \int_{t_0}^{t_1}|f'(t)|dt \right)\right.\\
& + & \sum_{k=1}^{n-1} w^{\alpha,\beta}(t_k) \left(\int_{t_k}^{t_{k+1}}|f(t)|dt+\Delta t_k \int_{t_k}^{t_{k+1}}|f'(t)|dt \right)\\
& + & w^{\alpha,\beta}(t_n)\left(\int_{t_{n-1}}^{t_n}|f(t)|dt+\Delta t_{n-1} \int_{t_{n-1}}^{t_n}|f'(t)|dt \right)\\
& + & w^{\alpha,\beta}(t_n)\left(\int_{t_n}^{t_{n+1}}|f(t)|dt+\Delta t_{n} \int_{t_n}^{t_{n+1}}|f'(t)|dt \right),
\end{eqnarray*}
from which, being for $k=0,1,\ldots,n$,
\begin{equation} \label{tilde1menox}
1\pm t_k \sim 1 \pm t \sim 1 \pm t_{k+1}, \qquad t_k \leq t \leq t_{k+1}, \quad
\end{equation}
and, taking into account \eqref{Deltaxnk} and also that
\begin{eqnarray*}
\Delta t_0 & \sim & \frac{\sqrt{1-t^2}}{n}, \qquad -1 < t \leq t_1, \\
\Delta t_n & \sim & \frac{\sqrt{1-t^2}}{n}, \qquad t_n \leq t < 1,
\end{eqnarray*}
 we deduce
\begin{eqnarray*}
\sum_{k=0}^{n+1} w_k|f(t_k)|& \leq & \C \left[\int_{t_0}^{t_1}|f(t)|w^{\alpha,\beta}(t)dt+\frac{1}{n}\int_{t_0}^{t_1}|f'(t)|\varphi(t)w^{\alpha,\beta}(t)dt \right.\\
& + & \sum_{k=1}^{n-1} \left(\int_{t_k}^{t_{k+1}}|f(t)|w^{\alpha,\beta}(t)dt+\frac{1}{n}\int_{t_k}^{t_{k+1}}|f'(t)|\varphi(t)w^{\alpha,\beta}(t)dt \right)\\
& + & \left.\int_{t_n}^{t_{n+1}}|f(t)|w^{\alpha,\beta}(t)dt+\frac{1}{n}\int_{t_n}^{t_{n+1}}|f'(t)|\varphi(t)w^{\alpha,\beta}(t)dt\right]\\
& = & \C \left[\int_{-1}^{1}|f(t)|w^{\alpha,\beta}(t)dt+\frac{1}{n}\int_{-1}^{1}|f'(t)|\varphi(t)w^{\alpha,\beta}(t)dt  \right]
\end{eqnarray*}
i.e. \eqref{firstinequality}. Since the $(n+1)$ quadrature formula \eqref{LobattoRule} is exact for any polynomial of degree at most $2n+1$, for $P \in \PP_{2n+1}$ we have
\begin{eqnarray*}
|e_n(f)| &=&|e_n(f-P)| \\
& \leq & \left|\int_{-1}^1(f-P)(t)w^{\alpha,\beta}(t)dt\right|+\left|\sum_{k=0}^{n+1}w_k(f-P)(t_k)\right|\\
& \leq & \|(f-P)w^{\alpha,\beta}\|_1+\sum_{k=0}^{n+1}w_k\left|(f-P)(t_k)\right|.
\end{eqnarray*}
Now, combining the previous inequality with \eqref{firstinequality}  and the following relation (see, for instance, \cite[p. 339]{MMbook}, \cite[p. 286]{MasRus})
\begin{equation*}
\|(f-P)'\varphi w^{\alpha,\beta} \|_1 \leq \C (2n+2) \|(f-P)w^{\alpha,\beta}\|_1 + \C_1 E_{2n}(f')_{\varphi w^{\alpha,\beta},1},
\end{equation*}
we can deduce
\begin{eqnarray*}
|e_n(f)| & \leq & \C\|(f-P)w^{\alpha,\beta}\|_1+\frac{\C}{n}\|(f-P)'\varphi w^{\alpha,\beta} \|_1\\
& \leq & \C\|(f-P)w^{\alpha,\beta}\|_1+\frac{\C_1}{n}E_{2n}(f')_{\varphi w^{\alpha,\beta},1}.
\end{eqnarray*}
Taking the infimum over $P \in \PP_{2n+1}$ and using the Favard inequality \eqref{Favard} we get
\begin{equation} \label{estimater1}
|e_n(f)| \leq \C E_{2n+1}(f)_{w^{\alpha,\beta},1}+\frac{\C_1}{n}E_{2n}(f')_{\varphi w^{\alpha,\beta},1} \leq \frac{\C}{2n}E_{2n} (f')_{\varphi w^{\alpha,\beta},1}.
\end{equation}
The estimate \eqref{QuadrErrorEst} for $r>1$ can be deduced from \eqref{estimater1}  by iterating the application of the Favard inequality.
\end{proof}

\section*{Acknowledgements}
The author is partially supported by University of Basilicata (local funds). The author is member of the INdAM Research group GNCS and the TAA-UMI Research Group. This research has been accomplished within ``Research ITalian network on Approximation'' (RITA).

\bibliographystyle{model1b-num-names}
\bibliography{Lauritabib}

\end{document}